\documentclass[12pt, reqno]{amsart}

\usepackage{amsmath,amscd, float,rotating}
\usepackage{pb-diagram}
\usepackage{latexsym}
\usepackage{amsfonts,color}
\usepackage{amsmath}
\usepackage{amssymb}
\usepackage{txfonts}
\usepackage{systeme}  
\usepackage{todonotes}

\newcommand{\R}{\mathbb{R}}

\newcommand{\N}{\mathbb{N}}
\newcommand{\Z}{\mathbb{Z}}
\newcommand{\cO}{\mathcal{O}}
\newcommand{\cM}{\mathcal{M}}
\newcommand{\cH}{\mathcal{H}}
\newcommand{\cE}{\mathcal{E}}

\newcommand{\pr}{\mathbb{P}}

\newcommand{\ddbar}{\sqrt{-1} \partial \overline{\partial}}

\newcommand{\om}{\omega}

\DeclareMathOperator{\ord}{ord}

\newtheorem{theorem}{Theorem}[section]
\newtheorem{lemma}[theorem]{Lemma}
\newtheorem{corollary}[theorem]{Corollary}
\newtheorem{proposition}[theorem]{Proposition}

\numberwithin{equation}{section}

\theoremstyle{definition}
\newtheorem{remark}[theorem]{Remark}

\theoremstyle{definition}

\begin{document}

\title{Constant scalar curvature K\"ahler metrics on ramified Galois coverings}

\begin{abstract}

 We give sufficient conditions for the existence of K\"ahler-Einstein and constant scalar curvature K\"ahler (cscK) metrics on finite ramified Galois coverings of a cscK manifold in terms of cohomological conditions on the K\"ahler classes and the branching divisor. This result generalizes  previous work on K\"ahler-Einstein metrics by Li-Sun \cite{LS}, and extends Chen-Cheng's existence results for cscK metrics in \cite{CC}. 

\end{abstract}

\author{Claudio Arezzo}
\address{ICTP and Universit\`a di Parma}
\email{arezzo@ictp.it}
\author{Alberto Della Vedova}
\address{Universit\`a di Milano - Bicocca}
\email{alberto.dellavedova@unimib.it}
\author{Yalong Shi}
\address{NJU}
\email{shiyl@nju.edu.cn}
\maketitle
\thispagestyle{empty}

\tableofcontents

\section{Introduction}

In this paper we provide a new geometric construction of K\"ahler-Einstein and constant scalar curvature K\"ahler (cscK) metrics on finite ramified Galois coverings of cscK manifolds. In order to state our results let us first recall the notion of a Galois cover.

Let $M, N$ be complex manifolds of dimension $n$. A holomorphic map $f:M\to N$ is called ``finite", if it is proper and has finite fibers. If $f$ is moreover surjective, it is called an ``analytic covering".  For such an analytic covering $f$, we can define the Galois group
$$Gal(f):=\{\gamma\in Aut(M)\ |\ f\circ\gamma=f\}.$$
$f$ is called a ``Galois covering" if any two points of $M$ lie on the same fiber of $f$ is and if they belong to the same $Gal(f)$-orbit. In other words, $N\cong M/Gal(f)$.

Now we assume both $M$ and $N$ are smooth compact K\"ahler manifolds, and $N$ has a constant scalar curvature K\"ahler metric (``cscK metric" for short) with K\"ahler form $\om_N$. If $f: M\to N$ is an unramified (or ``\'etale") Galois covering, then $f^*\om_N$ is a cscK metric on $M$. The central question addressed in this paper is whether (or when)  cscK metrics still exist in the class $[f^*\om_N]$ also when $f$ is ramified. Note that in this case while $f^*\om_N$ is degenerate, thanks to Demailly-Paun's result \cite{DeP}, the class $[f^*\om_N]$ is still a K\"ahler class (see Section $2$). 

When both $M$ and $N$ are Fano manifolds and $\om_N$ is K\"ahler-Einstein, this problem is well studied and partial resullts have been found by Arezzo-Ghigi-Pirola \cite{AGP}, Ghigi-Kollar \cite{GhKo} and Li-Sun \cite{LS}. Their methods are all variational, using the (non-trivial) fact that the properness of Ding-functional (resp. log-K-energy) implies the existence of K\"ahler-Einstein metrics (resp. conical K\"ahler-Einstein metrics).

A recent breakthrogh in K\"ahler geometry is Chen-Cheng's work \cite{CC}, showing that the properness of K-energy in a general K\"ahler class implies the existence of K\"ahler metrics of constant scalar curvature. In this paper, we shall use a ($G$-invariant version) of Chen-Cheng's theorem
(Theorem \ref{thm:ChenCheng}) to prove that the above mentioned theorems on K\"ahler-Einstein metrics generalize to the cscK setting.

To state our main theorem we introduce the following notations: let $f: M\to N$ be a finite Galois covering of degree $d$ with Galois group $G$. Denote the ramification divisor by $R_f=\sum_j r_j R_j$. It is known that locally on the regular part of $R_j$, $f$ is of the form $(z_1,\dots, z_n)\mapsto (z_1,\dots, z_{n-1}, z_n^{r_j+1})$ under appropriate holomorphic coordinates (Lemma 5 of \cite{GhKo}). So we have $K_M=f^*K_N+R_f$. As in \cite{GhKo}, we have $R_f=f^*B_f$ where $B_f=\sum_i\Big(1-\frac{1}{|\Gamma_i|}\Big)B_i$ is the branching divisor. Here $\Gamma_i:=\{\gamma\in Gal(f)\ |\ R_j\subset Fix(\gamma)\}$ is the inertia subgroup of any $R_j$ such that $f(R_j)=B_i$. Write $d_i:=|\Gamma_i|$.  Note that we do not assume $R_f$ or $B_f$ is simple normal crossing but we only require that both $M$ and $N$ are smooth.

The main result of this paper is then the following theorem:

\begin{theorem}\label{thm:Galois}
Let $(N,\om_N)$ be a $n$-dimensional compact K\"ahler manifold, and $f:M\to N$ a finite ramified Galois covering, branched along  divisor $B_f = \sum_{i=1}^k (1-\frac{1}{d_i})B_i$ of $N$ with ramification divisor $R_f$. Assume that the $K$-energy on $N$ of the class $[\om_N]$ is bounded from below (For example, if $[\om_N]$ has a cscK metric). Write $B_\delta:=\delta\sum_{i=1}^k (1-\frac{1}{d_i})B_i$ and assume that for some $\delta_0\in [1,\min\{\frac{d_i}{d_i-1}\})$, the classes $K_N+B_{\delta_0}$ and
$$\frac{n(K_N+B_{\delta_0})\cdot [\om_N]^{n-1}}{[\om_N]^n}[\om_N]-(n-1)(K_N+B_{\delta_0})$$
are both nef. Then there  exists a cscK metric on $M$ in the K\"ahler class $[f^*\om_N]$.
\end{theorem}

The simplest examples of ramified Galois coverings are finite cyclic coverings branching along an irreducible divisor (which is necessarily smooth by Catanese \cite[Proposition 1.1]{Cat}). In this special case, Theorem \ref{thm:Galois} reduces to:

\begin{corollary}\label{coro:cyclic}
Let $(N,\om_N)$ be a $n$-dimensional compact cscK manifold, and $f:M\to N$ a finite cyclic cover of order $d$, branched along an irreducible divisor $B\subset N$ with ramification divisor $R_f$. If the K\"ahler class $[\om_N]$ and $[B]$ satisfy the condition that there is a positive constant $\beta_0\in (0,\frac{1}{d}]$ such that both $K_N+(1-\beta_0)B$ and
 $$\frac{n\big(K_N+(1-\beta_0)B\big)\cdot [\om_N]^{n-1}}{[\om_N]^n}[\om_N]-(n-1)\big(K_N+(1-\beta_0)B\big)$$ are nef,
then there  exists a cscK metric on $M$ in the K\"ahler class $[f^*\om_N]$.
\end{corollary}

The idea of the proof to Theorem \ref{thm:Galois} is clearest in the case of simple cyclic cover, i.e. Corollary \ref{coro:cyclic}. In this case, one should have cscK metrics with cone angle $2\pi\beta$ along $B$ for any $\beta_0\leq\beta<1$. Then the pull back of the conical cscK metric with angle $\frac{2\pi}{d}$ will give rise to a smooth cscK metric in the class $[f^*\om_N]$. This is indeed true in the conical K\"ahler-Einstein case as observed in \cite{LS}. Even though there are already systematic studies on conic cscK metrics (e.g. \cite{Zh}), we can in fact avoid such ``conical cscK metrics" altogether except the concepts of ``log $K$-energy" and ``log-$\alpha$-invariant". 

Our strategy is as follows: 
\begin{align*}
\text{cscK on N}&\Longrightarrow \text{K-energy bounded from below on N}\\
&\Longrightarrow \text{log K-energy proper on N}\\
&\Longrightarrow \text{K-energy proper in the space of $G$-invariant potentials on M}.
\end{align*}

The first step is already well known. The main part of our proof is devoted to step 2 and step 3. 

\begin{remark}
When $B_f$ is simple normal crossing, one can use Zheng's theorem on conical cscK metrics \cite{Zh}. Our method avoids analytical difficulties involving conical metrics. Moreover, our theorem works even if $B_f$ is not simple normal crossing, while a suitable ``conical theory" is still missing.
\end{remark}

\begin{remark}
In view of the recent work of Z. Sj\"ostr\"om Dyrefelt \cite{SD} and K. Zhang \cite{Zha}, our main result can be improved. See, for example Remark \ref{rk:Zak}. However, we stick to this weaker version since the conditions here are much easier to check. 
\end{remark}

When $f$ is a cyclic covering branched over a smooth hypersurface $B$, and $[\om_N]$ is proportional to $c_1(B)$, we can strengthen Corollary \ref{coro:cyclic} using Berman's estimate of log-$\alpha$-invariants \cite{Ber}: 

\begin{theorem}\label{thm:cyclic}
Let $(N,\om_N)$ be a $n$-dimensional compact K\"ahler manifold, and $f:M\to N$ a finite cyclic Galois covering, branched along a smooth  hypersurface $B$ of $N$. Assume $\om_N$ is a cscK metric on $N$ with $[\om_N]=\lambda 2\pi c_1(B)$ for some $\lambda>0$. If both $K_N+B$ and 
$$\Big(\frac{nK_N\cdot B^{n-1}}{B^n}+1\Big)B-(n-1)K_N$$ are nef, then there exists a cscK metric on $M$ in the class $[f^*\om_N]$.
\end{theorem}

This contains as a special case Theorem 6.1 of Li-Sun \cite{LS}, which covers the case of $N$ Fano and $\om_N$ K\"ahler-Einstein. In fact, if  $B\in |-\beta K_N|$ with $\beta\geq 1$, both conditions are satisfied, we can apply Theorem \ref{thm:cyclic} to get the K\"ahler-Einstein metric on $M$.\\

The structure of the paper is as follows. In section \ref{sec:preliminaries}, we discuss basic properties of $K$-energy on the space of potentials with finite energies. Most crucial is Lemma \ref{lem:approx_E1}, saying that to check properness of the $K$-energy on $\mathcal{E}^1$ space, it suffices to check it on the space of smooth potentials. Then in section \ref{sec:K-energy}, we compute the K-energy in the space of $G$-invariant potentials and relate it to the log-K-energy downstairs. Theorem \ref{thm:Galois} is proved in section \ref{sec:Main}, using a log-version of Li-Shi-Yao's criterion for properness of the K-energy (which is independently discovered by K. Zheng in the conic setting), as well as an equivariant version of Chen-Cheng's theorem. This generalized version of Chen-Cheng's theorem is essentially contained in \cite{CC}. However, since it is neither stated explicitly, nor implied directly by their theorems, we decide to include a sketch of proof for the convenience of readers. Then in section \ref{sec:ThmCyclic}, making use of Berman's estimate of the log-$\alpha$ invariant, we prove Theorem \ref{thm:cyclic}.
In the last Section \ref{sec:example}, we provide large families of new examples both of K\"ahler-Einstein and constant scalar curvature K\"ahler (cscK) metrics and compare on explicit cases our results with other known existence Theorems, notably Li-Sun \cite{LS} and Chen-Cheng's \cite{CC} ones.

{\bf{Aknowledgements:}} The fist author is very grateful to Fabio Perroni for beautiful explanations of the Catanese-Perroni's theory largely used in the last Section to produce many interesting examples. The third author thanks K. Zheng for helpful discussions on conical metrics. This work began when the second and third authors were both visiting ICTP, they would like to thank ICTP for its hospitality.

\section{Preliminaries}\label{sec:preliminaries}

Before entering the analytical details, let us observe that 
the class $[f^*\om_N]$ is indeed a K\"ahler class even when $f$ has a non-empty ramification.
In fact, when $N$ is projective algebraic, this is standard result, see, for example, Proposition 1.2.13 of \cite{La}. In the general case, we can apply Demailly-Paun's Theorem \cite{DeP} generalizing the Nakai-Moishezon criterion, stating that hat the K\"ahler cone is one of the connected components of the set of all real $(1,1)$-class $\alpha$ such that $\int_V \alpha^{\dim V}>0$ for any analytic cycle $V$. In our case it is immediate to observe that 
$$\int_V (f^*\om_N)^p=d\int_{f(V)}\om_N^p>0,$$
since $f(V)$ is a $p$-dimensional analytic cycle of $N$.
Since $f^*\om_N\geq 0$, it is in the closure of the K\"ahler cone and hence in the K\"ahler cone.

Let us now introduce the spaces of K\"ahler potentials we shall use: let $X$ be a compact K\"ahler manifold and $\omega$ be a smooth non-negative closed $(1,1)$ form with $[\om]$ is a K\"ahler class. We define
$$\cH(X, \omega):=\{\varphi\in C^\infty(X;\R)\ |\ \om+\ddbar\varphi>0\},$$
and
$$PSH(X,\omega):=\{\varphi\in L^1(X;\R)\ |\ \om+\ddbar\varphi\geq 0\}.$$
Here $L^1(X;\R)$ is defined with respect to any smooth volume form, and ``$\om+\ddbar\varphi\geq 0$" is in the sense of currents.  Another important space is the space of potentials with finite energy, defined as follows: write $\om_\varphi:=\om+\ddbar\varphi$
$$\cE^1(X,\om):=\Big\{\varphi\in L^1(X;\R)\ \Big|\ \int_X \om_\varphi^n=\int_X\om^n,\quad \int_X |\varphi|\om_\varphi^n<+\infty\Big\},$$ 
where $\om_\varphi^n:=\lim_{j\to\infty} \mathbf{1}_{\varphi>-j}\Big(\om+\ddbar\max\{\varphi,-j\}\Big)^n$ is defined as the pluripolar product. 

Obviously, $\cH(X,\om)\subset PSH(X,\om)\cap C^0(X;\R)\subset \cE^1(X,\om)\subset PSH(X,\om)$. It is proved by Darvas that $\cE^1(X,\om)$ is the metric completion of $\cH(X,\om)$ with respect to the $L^1$-Finsler metric on $\cH$. The $\cE^1$ potentials appear naturally in our problem in view of the following lemma:

\begin{lemma}\label{lem:potential}
Let $f: M\to N$ be a Galois covering between compact K\"ahler manifolds with Galois group $G$ and branching locus $B_{red}$. Let $\om_N$ be a K\"ahler form on $N$. Then for any $G$-invariant $\varphi\in\cH(M,f^*\om_N)$, we can find a function $\phi\in C^0(N;\R)\cap C^\infty(N\setminus B_{red};\R)\cap PSH(N,\om_N)$ such that $\varphi=f^*\phi=\phi\circ f$. In particular, $\phi\in\cE^1(N,\om_N)$.
\end{lemma}

\begin{proof}
The existence of $\phi\in C^0(N;\R)\cap C^\infty(N\setminus B_{red};\R)$ such that $\varphi=f^*\phi=\phi\circ f$ is obvious. Also $\om_N+\ddbar\phi>0$ on $N\setminus B_{red}$. To show $\phi\in PSH(N,\om_N)$, just observe that the direct image of the K\"ahler form $f^*\om_N+\ddbar\varphi$ is precisely $d(\om_N+\ddbar\phi)$, which is a K\"ahler current by Lemma 2.6 of \cite{AGP}.
\end{proof}

For a $\mathbb{Q}$-divisor of the form $D=\sum_{i}(1-\beta_i)D_i$ on $N$, and let the defining section of $D_i$ be $s_i$ with metric $h_i$ on the bundle $\mathcal{O}(D_i)$. We can define the log-K-energy by
\begin{equation}
K_{D,\omega_N}(\varphi):=\int_N\log\frac{\om^n_{\varphi}}{\om^n_N}\frac{\om^n_{\varphi}}{n!}+\sum_i(1-\beta_i)\int_N \log \|s_i\|_{h_i}^2\Big(\frac{\om^n_{\varphi}}{n!}-\frac{\om^n_N}{n!}\Big)+\hat J_{-Ric(\om_N)+\sum_i(1-\beta_i)\Theta_{h_i}}(\varphi).
\end{equation}

\begin{lemma}\label{lem:approx_E1}
To check the properness of log-K-energy on the space of $\mathcal{E}^1$ potentials, it suffices to check the properness on the space of smooth potentials.
\end{lemma}

\begin{proof}
Suppose we have a uniform estimate of the form
$$K_{D,\om_N}(\varphi)\geq \epsilon (I-J)(\varphi)-C$$
on $\cH(N,\om_N)$.
Now let $\varphi\in \cE^1(N,\om_N)$ . We can find a sequence $\varphi_i\in\cH(N,\om_N)$ converging to $\varphi$ in the $\cE^1$-topology.

Now it is well-known that $(I-J)(\varphi_i)\to (I-J)(\varphi)$. We need to show that
$$K_{D,\om_N}(\varphi)\geq \liminf_{i\to\infty} K_{\beta,\om_N}(\varphi_i).$$
Then from
$$K_{D,\om_N}(\varphi_i)\geq \epsilon (I-J)(\varphi_i)-C$$
we will conclude that
$$K_{D,\om_N}(\varphi)\geq \epsilon (I-J)(\varphi)-C$$
holds for $\cE^1$ potential $\varphi$.
However, the K-energy is only lower semi-continuous with respect to the $d_1$-distance. So we need Berman-Darvas-Lu's result (Lemma 3.1 of \cite{BDL}): there is a sequence of smooth potentials $\varphi_k$ such that
$d_1(\varphi_k,\varphi)\to 0$ and $$\int_N\log\frac{\om^n_{\varphi_k}}{\om^n_N}\frac{\om^n_{\varphi_k}}{n!}\to\int_N\log\frac{\om^n_{\varphi}}{\om^n_N}\frac{\om^n_{\varphi}}{n!}, (k\to\infty).$$
Now both $I-J$ and $\hat J_{-Ric(\om_N)+\sum(1-\beta_i)\Theta_{h_i}}$ are continuous on $\mathcal{E}^1$ with respect to the $d_1$ distance (see, for example, Lemma 5.23 of \cite{DarRub}), so 
$$(I-J)(\varphi_k)\to (I-J)(\varphi),\quad \hat J_{-Ric(\om_N)+\sum(1-\beta_i)\Theta_{h_i}}(\varphi_i)\to \hat J_{-Ric(\om_N)+\sum(1-\beta_i)\Theta_{h_i}}(\varphi),\quad (k\to\infty).$$
Finally we need to cope with the terms $\int_N \log \|s_i\|_{h_i}^2\frac{\om^n_{\varphi}}{n!}$. For simplicity, we omit the subscript $i$ in the following arguments.  We shall use the fact the $d_1$ convergence implies the weak convergence of Monge-Amper\`e measures (Corollary 5.8 of \cite{Dar}). So for the continuous function $\log (\|s\|_h^2+\delta)$ ($\delta>0$), we have
$$\int_N \log (\|s\|_h^2+\delta)\frac{\om^n_{\varphi_k}}{n!}\to \int_N \log (\|s\|_h^2+\delta)\frac{\om^n_{\varphi}}{n!}, \quad (k\to\infty).$$
So we conclude that
\begin{eqnarray*}
& & K_{D,\omega_N}(\varphi)+\sum_i(1-\beta_i)\int_N \Big(\log (\|s_i\|_{h_i}^2+\delta)-\log \|s_i\|_{h_i}^2\Big)\frac{\om^n_{\varphi}}{n!}\\
&=& \lim_{k\to\infty} \Big(K_{D,\omega_N}(\varphi_k)+\sum_i(1-\beta_i)\int_N \Big(\log (\|s_i\|_{h_i}^2+\delta)-\log \|s_i\|_{h_i}^2\Big)\frac{\om^n_{\varphi_k}}{n!}\Big)\\
&\geq& \liminf_{k\to\infty}K_{\beta,\omega_N}(\varphi_k)\\
&\geq &\liminf_{k\to\infty} \epsilon (I-J)(\varphi_k)-C\\
&=& \epsilon (I-J)(\varphi)-C.
\end{eqnarray*}
Finally, let $\delta\to 0$ and use the monotone convergence theorem, we get $K_{D,\om_N}(\varphi)\geq \epsilon (I-J)(\varphi)-C$.
\end{proof}

\section{K-energy under Galois coverings}\label{sec:K-energy}

Let $f: M\to N$ be a finite Galois covering of degree $d$ with Galois group $G$ with ramification divisor  $R_f$ and branching divisor $B_f=\sum_i\Big(1-\frac{1}{d_i}\Big)B_i$. To study the behavior of K-energy under Galois coverings, it is convenient to use the degenerate reference form $f^*\om_N$ as the reference form in the definition of K-energy. This causes no trouble, since for smooth $\varphi$ satisfying $\om_{\varphi}:=\om_0+\ddbar\varphi>0$, set $\psi:=\varphi+\varphi_0$, then for any $\epsilon>0$ and sufficiently small, we have
$f^*\om_N+\epsilon\ddbar\varphi_0=(1-\epsilon)f^*\om_N+\epsilon\om_0>0$. Since the K-energy satisfies the cocycle conditions, we get
\begin{eqnarray*}
K_{\om_0}(\varphi)&=:&\cM(\om_0,\om_\varphi)=\cM(\om_0,f^*\om_N+\epsilon\ddbar\varphi_0)-\cM(\om_\varphi,f^*\om_N+\epsilon\ddbar\varphi_0)\\
&=& K_{\om_0}(-(1-\epsilon)\varphi_0)-K_{\om_\varphi}(-\psi+\epsilon\varphi_0).
\end{eqnarray*}
Using Chen-Tian's formula, it is easy to see that when $\epsilon\to 0$, both terms have finite limits: the entropy term converges by bounded convergence theorem and the convergence of the $\hat J$ term is trivial.  Then we can write this as a generalized cocycle condition:
\begin{eqnarray*}
K_{\om_0}(\varphi)=K_{\om_0}(-\varphi_0)+K_{f^*\om_N}(\psi),
\end{eqnarray*}
where we define $K_{f^*\om_N}(\psi)$ as
$$K_{f^*\om_N}(\psi):=-K_{\om_\varphi}(-\psi):=-\lim_{\epsilon\to 0+}K_{\om_\varphi}(-\psi+\epsilon\varphi_0).$$
So the properness of $K_{\om_0}$ on $\cH(\om_0)$ is equivalent to the properness of $K_{f^*\om_N}$ on $\cH(f^*\om_N):=\{\psi\in C^\infty(M;\R)\ |\ f^*\om_N+\ddbar\psi>0\}$.

The key observation of our work is the following lemma, generalizing Lemma 2.7 of \cite{AGP} for the Ding-functional and Proposition 6.2 of \cite{LS} for the K-energy.
\begin{lemma}\label{lem:K for Galois}
Let $f:M\to N$ be as above. If $\varphi\in \cH(M,f^*\om_N)$ is $G$-invariant, then we have
\begin{equation}\label{eqn:K for Galois}
    K^M_{f^*\om_N}(\varphi)=d\cdot K^N_{B_f,\om_N}(\phi),
\end{equation}
where $\phi$ is the function on $N$ such that $\varphi=f^*\phi$.
\end{lemma}

\begin{proof}
Let us write $\omega_{t\varphi}:=f^*\omega_N+t\ddbar\varphi$ and $\omega_{t\phi}:=\omega_N+t\ddbar\phi$, then we have
\begin{eqnarray*}
K^M_{f^*\omega_N}(\varphi)&=&-\int_0^1 dt\int_M\varphi \big(s(\omega_{t\varphi})-\underline{s}^M\big) \frac{\omega_{t\varphi}^n}{n!}\\
&=& \int_0^1 dt\int_M\varphi\Big[\ddbar\log\frac{\omega_{t\varphi}^n}{f^*\omega_N^n}+\ddbar\log f^*\omega_N^n\Big]\wedge\frac{\omega_{t\varphi}^{n-1}}{(n-1)!}+\underline{s}^M\int_0^1 dt\int_M \varphi\frac{\omega_{t\varphi}^n}{n!}\\
&=& \int_M \log\frac{\omega_{\varphi}^n}{f^*\omega_N^n}\frac{\omega_{\varphi}^n}{n!}+\int_0^1 dt\int_M \varphi\Big[\frac{\underline{s}^M}{n}\omega_{t\varphi}-f^*Ric(\omega_N)\Big]\wedge\frac{\omega_{t\varphi}^{n-1}}{(n-1)!}\\
& &+\sum_j 2\pi r_j\int_0^1 dt\int_{R_j}\varphi\frac{\omega_{t\varphi}^{n-1}}{(n-1)!}\\
&=& d\Big(\int_N \log\frac{\omega_{\phi}^n}{\omega_N^n}\frac{\omega_{\phi}^n}{n!}+\int_0^1 dt\int_N \phi\Big[\frac{\underline{s}^M}{n}\omega_{t\phi}-Ric(\omega_N)\Big]\wedge\frac{\omega_{t\phi}^{n-1}}{(n-1)!}\Big)\\
& &+\sum_i\sum_{j: f(R_j)=B_i} 2\pi r_j \int_0^1 dt\int_{R_j}\varphi\frac{\omega_{t\varphi}^{n-1}}{(n-1)!}.
\end{eqnarray*}
On one hand,
\begin{eqnarray*}
\underline{s}^M &=&\frac{2\pi n c_1(M)\cdot[f^*\omega_N]^{n-1}}{[f^*\omega_N]^n}=\frac{2n\pi f^*(c_1(N)-[B_f])\cdot [f^*\omega_N]^{n-1}}{[f^*\omega_N]^n}\\
&=&\frac{2n\pi (c_1(N)-[B_f])\cdot [\omega_N]^{n-1}}{[\omega_N]^n}=\underline{s}^N-\sum_i\Big(1-\frac{1}{|\Gamma_i|}\Big)\frac{2n\pi [B_i]\cdot[\omega_N]^{n-1}}{[\omega_N]^n}\\
&=&\underline{s}^N-\sum_i\Big(1-\frac{1}{|\Gamma_i|}\Big)\frac{n [\Theta_{h_i}]\cdot[\omega_N]^{n-1}}{[\omega_N]^n}.
\end{eqnarray*}

On the other hand, for any $R_j$ such that $f(R_j)=B_i$, suppose $f|_{R_j}: R_j\to B_i$ has degree $d_j\geq 1$. Since $f$ is Galois, both $d_j$ and $r_j=|\Gamma(R_j)|-1$  depends only on $B_i$, not on $R_j$, so we can just write $d_i$ and $r_i=|\Gamma_i|-1$. Then we have
\begin{eqnarray*}
\sum_i\sum_{j: f(R_j)=B_i} 2\pi r_j \int_0^1 dt\int_{R_j}\varphi\frac{\omega_{t\varphi}^{n-1}}{(n-1)!}=\sum_i 2\pi(|\Gamma_i|-1)d_ik_i\int_0^1 dt\int_{B_i}\phi\frac{\omega_{t\phi}^{n-1}}{(n-1)!},
\end{eqnarray*}
where $k_i$ is the number of different $R_j$'s mapped onto $B_i$.

Now a generic point $q\in N$ should have $d$ pre-images. Choose a point near $B_i$, we will have $d_i(r_i+1)=d_i|\Gamma_i|$ pre-images near $R_j$, so we get $d=k_id_i|\Gamma_i|$. So we get 
\begin{eqnarray*}
\sum_i\sum_{j: f(R_j)=B_i} 2\pi r_j \int_0^1 dt\int_{R_j}\varphi\frac{\omega_{t\varphi}^{n-1}}{(n-1)!}&=&\sum_i 2\pi(|\Gamma_i|-1)\frac{d}{|\Gamma_i|}\int_0^1 dt\int_{B_i}\phi\frac{\omega_{t\phi}^{n-1}}{(n-1)!}\\
&=& d\sum_i\Big(1-\frac{1}{|\Gamma_i|}\Big)\int_0^1 dt\int_N \phi\Big[\ddbar\log \|s_i\|_{h_i}^2+\Theta_{h_i}\Big]\wedge\frac{\omega_{t\phi}^{n-1}}{(n-1)!}.
\end{eqnarray*}
Also note that
\begin{eqnarray*}
\int_0^1 &dt&\int_N \phi\ddbar\log \|s_i\|_{h_i}^2\wedge\frac{\omega_{t\phi}^{n-1}}{(n-1)!}= \int_0^1 dt\int_N \log \|s_i\|_{h_i}^2\frac{d}{dt}\frac{\omega_{t\phi}^{n}}{n!}\\
&=&\Big(\int_N \log \|s_i\|_{h_i}^2\frac{\omega_{t\phi}^{n}}{n!}\Big)\Big|_{t=0}^{t=1}=\int_N \log \|s_i\|_{h_i}^2\Big(\frac{\omega_{\phi}^{n}}{n!}-\frac{\omega_{N}^{n}}{n!}\Big)
\end{eqnarray*}
From these identities we at once get \eqref{eqn:K for Galois}.
\end{proof}

\section{Proof of Theorem \ref{thm:Galois}}\label{sec:Main}

The following theorem is essentially the same as Chen-Cheng \cite{CC}. Since it is not stated explicitly, we include a sketched proof here for the convenience of the readers: 
\begin{theorem}\label{thm:ChenCheng}
If K-energy is proper on the space of $G$-invariant K\"ahler potentials for a finite subgroup $G$ of $Aut(N,J)$, then there is a cscK metric.
\end{theorem}

\begin{proof}
We follow the same strategy as Tian did in \cite{T0}, using Chen-Cheng's continuity method. Start with a $G$-invariant K\"ahler form $\omega_0$ in the given $G$-invariant K\"ahler class. We use Chen's continuity path proposed in \cite{Chen}:
$$t(s(\omega_{\varphi})-\underline s)=(1-t)(tr_{\omega_{\varphi}}\omega_0-n).$$
In \cite{Chen}, Chen proved that the set of $t\in (0,1)$ such that the above equation is solvable is open, and the openness at $t=0$ is proved later by Hashimoto \cite{Has} and Zeng \cite{Zeng}. Their arguments work also for $G$-invariant function spaces when $G$ is finite. Then we apply Chen-Cheng's {\it a priori} estimates in \cite{CC} to get the closedness. 
\end{proof}

Given this theorem, to prove Theorem \ref{thm:Galois}, it suffices to prove the properness of $K^M_{f^*\omega_N}$ on the space of $G$-invariant potentials $\cH^G(M,f^*\omega_N)$. In turn, by Lemma \ref{lem:K for Galois} and \ref{lem:potential}, it suffices to prove the properness of $K^N_{B_f,\omega_N}$ on the space of $\cE^1$ potentials $\cE^1(N,\omega_N)$ with $B_f=\sum_i(1-\beta_i)B_i:=\sum_i(1-\frac{1}{|\Gamma_i|})B_i$. Finally by Lemma \ref{lem:approx_E1}, it suffices to prove the properness of log-K-energy on $\cH(N,\om_N)$.
The idea, as in \cite{LSY} is again to use the log-$\alpha$-invariant:
$$\alpha(\om_N, D):=\sup\{\alpha>0\big| \int_N e^{-\alpha(\varphi-\sup\varphi)}\frac{\om_N^n}{\Pi_i|s_i|_{h_i}^{2(1-\beta_i)}n!}\leq C_\alpha\quad \forall \varphi\in PSH(N,\om_N)\}.$$
Also one can prove that $\alpha(\om_N, D)>0$.

We have the following lemma, which is first observed by K. Zheng (Proposition 4.2 of \cite{Zh}) when $D$ is simple normal crossing. The proof here is essentially the same. For the convenience of readers, we include a sketched proof here.
\begin{lemma}\label{lem:general log LSY}
Let $D=\sum_i(1-\beta_i)D_i$. Assume we can find a constant $\epsilon\geq 0$ such that
\begin{enumerate}
    \item $\alpha(\om_N, D)>\frac{n\epsilon}{n+1}$;
    \item $2\pi (K_N+D)+\epsilon[\om_N]>0$;
    \item $$\Big(\frac{2n\pi(K_N+D)\cdot [\om_N]^{n-1}}{[\om_N]^n}+\epsilon\Big)[\om_N]>2\pi(n-1)\big{(}K_N+D\big{)},$$
\end{enumerate}
then $K^N_{D,\omega_N}$ is proper.
\end{lemma}

\begin{proof}

Choose $\frac{n\epsilon}{n+1}<\alpha<\alpha(\om_N, D)$. Then by definition, we have
\begin{align*}
    C_\alpha &\geq \int_N e^{-\alpha(\varphi-\sup\varphi)}\frac{\om_N^n}{\Pi_i|s_i|_{h_i}^{2(1-\beta_i)}n!}\\
    &=\int_N e^{-\alpha(\varphi-\sup\varphi)-\sum_i(1-\beta_i)\log|s_i|_{h_i}^2-\log\frac{\om_\varphi^n}{\om_N^n}}\frac{\om_\varphi^n}{n!}.
\end{align*}
By Jensen inequality, we get
\begin{align*}
 & \int_N \log\frac{\om_\varphi^n}{\om_N^n}\frac{\om_\varphi^n}{n!}+\sum_i(1-\beta_i)\int_N \log|s_i|_{h_i}^2 \frac{\om_\varphi^n}{n!} \\
 & \geq \alpha\int_N (\sup\varphi-\varphi)\frac{\om_\varphi^n}{n!}-C\\
 &\geq \alpha\int_N \varphi \Big(\frac{\om_N^n}{n!}-\frac{\om_\varphi^n}{n!}\Big)-C=\alpha I_{\om_N}(\varphi)-C\\
 &\geq \frac{(n+1)\alpha}{n}(I-J)(\varphi)-C.
\end{align*}
Using the above inequality, we have
$$K^N_{D,\om_N}(\phi)\geq \Big( \frac{(n+1)\alpha}{n}-\epsilon\Big)(I-J)(\varphi)+\hat J_{-Ric(\om_N)+\sum_i(1-\beta_i)\Theta_{h_i}+\epsilon\om_N, \om_N}(\varphi)-C.$$
Write $\chi_\epsilon:=-Ric(\om_N)+\sum_i(1-\beta_i)\Theta_{h_i}+\epsilon\om_N$, then condition (2) implies $[\chi_\epsilon]>0$. Now condition (3) implies the solvability of the corresponding $J$-equation, so $\hat J_{\chi_\epsilon, \om_N}$ is bounded from below.
\end{proof}

\begin{remark}\label{rk:Zak}
In the above proof, we can write instead
$$K^N_{D,\om_N}(\phi)\geq \Big( \frac{(n+1)\alpha}{n}\Big)(I-J)(\varphi)+\hat J_{-Ric(\om_N)+\sum_i(1-\beta_i)\Theta_{h_i}, \om_N}(\varphi)-C.$$
Set $\theta:=-Ric(\om_N)+\sum_i(1-\beta_i)\Theta_{h_i}$, then $[\theta]=2\pi(K_N+D)$. Using the work of Sj\"ostr\"om Dyrefelt \cite{SD},  we have
$$\sup\{\delta\in\R |\ \hat J_{\theta,\om_N}\geq \delta(I-J)_{\om_N}-C_\delta\}\geq \min\Big\{\mathcal{T}([\theta],[\om_N]), \inf_V\frac{C_{\theta,\om_N}\int_V \om_N^p-p\int_V \theta\wedge\om_N^{p-1}}{(n-p)\int_V\om_N^p}\Big\},$$
where $\mathcal{T}([\theta],[\om_N]):=\sup\{\delta\in\R|\ [\theta]-\delta[\om_N]\geq 0\}$, $C_{\theta,\om_N}:=\frac{n\int_N \theta\wedge \om_N^{n-1}}{\int_N \om_N^n}$ and the inf is taken with respect all subvarieties of $N$ of dimension $p\leq n-1$. 

Then Lemma \ref{lem:general log LSY} can be strengthened to:
{\em If
$$\frac{n+1}{n}\alpha(\om_N, D)+\min\Big\{\mathcal{T}\big([\theta],[\om_N]\big),\  \inf_V\frac{C_{\theta,\om_N}\int_V \om_N^p-p\int_V \theta\wedge\om_N^{p-1}}{(n-p)\int_V\om_N^p}\Big\}>0,$$
where $[\theta]=2\pi(K_N+D)$, then $K^N_{D,\omega_N}$ is proper.} In view of K. Zhang's work on the analytic $\delta$-invariant, we can even replace the log-$\alpha$ invariant by certain $\delta$-invariant \cite{Zh}.
\end{remark}

Now we can finish the proof of Theorem \ref{thm:Galois}:
\begin{proof}[Proof of Theorem \ref{thm:Galois}:]

By Lemmma \ref{lem:general log LSY}, the assumptions in Theorem \ref{thm:Galois} imply that $K_{B_{\delta_0},\om_N}$ is proper on the space of smooth K\"ahler potentials. By the linearity of $K_{B_\delta,\om_N}$ on $\delta$, we get, for $0<\delta\leq \delta_0$,
$$K_{B_\delta,\om_N}=\frac{\delta}{\delta_0}K_{B_{\delta_0},\om_N}+(1-\frac{\delta}{\delta_0})K_N.$$
Since $K_N$ is bounded from below, $K_{B_{\delta},\om_N}$ is proper on the space of smooth K\"ahler potentials for all $\delta\in (0,\delta_0]$. Take $\delta=1$, we get the properness of $K_{B_f,\om_N}$ on the space of smooth K\"ahler potentials. 
This implies the properness of $K_M$ on the space of smooth $G$-invariant K\"ahler potentials and hence the existence of cscK metrics in $[f^*\om_N]$.
\end{proof}

\section{Proof of Theorem \ref{thm:cyclic}}\label{sec:ThmCyclic}

Now let $f:M\to N$ be a cyclic Galois covering of order $d$ branched over a smooth hypersurface $B\subset N$, where both $M$ and $N$ are project algebraic manifolds of dimension $n$. Moreover, we assume that $N$ has a cscK metric with $[\om_N]= \lambda 2\pi c_1(B)$ for some $\lambda>0$. We can make use of the following estimate of Berman \cite{Ber} for the log-$\alpha$-invariant:
\begin{lemma}[\cite{Ber} Proposition 6.2]
Let $N$ be a projective algebraic manifold with $B\subset N$ a smooth hypersurface. Let $0<\beta<1$, then we have\begin{equation*}
    \alpha(2\pi c_1(B), (1-\beta)B)\geq \min\big\{\beta, \alpha(2\pi c_1(B)), \alpha(2\pi c_1(B)|_B)\big\}.
\end{equation*}
\end{lemma}
From this we get
\begin{equation}
    \alpha([\om_N],(1-\beta)B)\geq \min\big\{ \frac{\beta}{\lambda}, \alpha([\om_N]), \alpha([\om_N]|_B)\big\}.
\end{equation}
To show the properness of $K_{(1-\beta)B,\om_N}$, we apply Lemma \ref{lem:general log LSY}: we need to find a number $\epsilon\geq 0$ such that
\begin{enumerate}
    \item $\alpha(\om_N, (1-\beta)B)>\frac{n\epsilon}{n+1}$;
    \item $K_N+(1-\beta+\lambda\epsilon)B>0$;
    \item $$\Big(\frac{nK_N\cdot B^{n-1}}{B^n}+1-\beta+\lambda\epsilon\Big)B-(n-1)K_N>0.$$
\end{enumerate}
When $\beta>0$ is sufficiently small, for condition (1), we only need to choose $\epsilon$ such that 
$$\frac{n\epsilon}{n+1}<\frac{\beta}{\lambda}.$$
So in case both $K_N+B$ and
$$\Big(\frac{nK_N\cdot B^{n-1}}{B^n}+1\Big)B-(n-1)K_N$$
are nef, we can take $\epsilon$ slightly bigger than $\frac{\beta}{\lambda}$ such that both (2) and (3) are true. This proves that $K_{(1-\beta)B,\om_N}$ is proper for sufficiently small $\beta>0$. On the other hand, when $\beta=1$, $K_{\om_N}$ is bounded from below. The remaining arguments are identical to the proof of Theorem \ref{thm:Galois}.

\section{Applications and Examples}\label{sec:example}

\subsection{Existence of covers}
In order to discuss the range of applicability of our results, the first point to address is certainly the existence problem for Galois covers between smooth complex manifolds.

The simplest and most classical situation is the so called ``simple cyclic cover" (see e.g.
Lazarsfeld \cite{La} Proposition 4.1.6): let $N$ be a smooth K\"ahler manifold and $L$ a holomorphic line bundle over $N$, with total space $\mathbf{L}$. Suppose there is a holomorphic section $\sigma\in H^0(N,\cO(L^{\otimes d}))$ ($d\in\N$ and $d\geq 2$) such that $(\sigma=0)$ is a smooth prime divisor. Then one can define $$M:=\{(p,v)\in \mathbf{L}\ |\ v^{\otimes d}=\sigma(p)\}.$$
Then it is easy to check that $M$ is a smooth K\"ahler manifold and the restriction of the projection map of $\mathbf{L}$ is a cyclic cover of degree $d$, branched along $(\sigma=0)$. 

One should note concerning this type of construction that the smoothness of $(\sigma=0)$ suffices to prove smoothness of $M$,
but it is clearly in general not necessary. 
Moving to the general abelian case this relationship has been well understood. Indeed on one hand,
given a branched covering $f:M \to N$ with $M$ and $N$ smooth, and abelian Galois group $Gal(f)$, Catanese (\cite[Proposition 1.1]{Cat}) proved that both the ramification divisor $R_f$ and the branching divisor $B_f$ are simple normal crossing.

Viceversa, given a smooth base $N$ and a simple normal crossing divisor $B$ in $N$, Bloch-Gieseker-Kawamata proved the existence of a smooth abelian Galois cover with {\em{some}} Galois group $G$. (see e.g.
Lazarsfeld \cite{La} Proposition 4.1.12 and Theorem 1.1.1 in Kawamata-Matsuda-Matsuki \cite{KMM}).

A complete general theory for the existence of Galois covers with preassigned abelian Galois group $G$ has been 
developed in the seminal paper by Pardini \cite{P}.

Concerning the general abelian case, let us observe that thanks to the above mentioned result by Catanese, the inertia group $G_j$ of each component $R_j$ of $R_f$, i.e. the intersection $G_j = \bigcap_{x \in R_j} G_x$ of all the isotropy groups of points of $R_j$, is a cyclic group. Also, the isotropy group $G_x = \{g \in G \,:\, gx=x\}$ of a point $x \in M$ is the direct product of the inertia groups of the irreducible components of $R_f$ passing through $x$ .
Furthermore, note that given a component $B_i$ of $B_f$, all the components of $f^{-1}(B_i)$ have the same inertia group since $G$ is abelian. 

As a consequence, one can deduce that an abelian Galois covering $f:M \to N$ factorizes as a composition of branched cyclic coverings and an unramified abelian covering (as indeed happens for the above mentioned Kawamata's coverings), a result which seems to be of independent interest and makes the production of further cscK examples straightforward.

\begin{proposition}
Let $f:M \to N$ be an abelian Galois covering, and put $M_0=M$.
There exist cyclic subgroups $G_1,\dots,G_\ell \subset Gal(f)$ together with Galois $G_j$-coverings $f_j : M_{j-1} \to M_j$ for all $j \in \{1,\dots,\ell\}$, and an unramified $\hat G$-covering $\hat f : M_\ell \to N$, where $\hat G = Gal(f)/\prod_{j=1}^\ell G_j$, such that
\begin{equation*}
    f = \hat f \circ f_\ell \circ \dots \circ f_1.
\end{equation*}
\end{proposition}
\begin{proof}
Let $K = \prod_{x \in M} G_x \subset G$ be the subgroup generated by all the isotropy groups for the action of $G$ on $M$. 
Since $K$ is a finite abelian group, there exists a subset $S \subset K$ such that $K = \bigoplus_{g \in S} \langle g \rangle$, where $ \langle g \rangle \simeq \mathbb Z_{\ord(g)}$ is the cyclic group generated by $g$, and the order $\ord(g)$ of $g$ is a power of a prime.

Pick $g \in S$, and let $\ord(g)=p^m$ for some prime $p$ and integer $m$.
Note that, in general, $Fix(g) = \{x \in M : gx=x \}$ may be empty. 
On the other hand, since $\langle g \rangle$ is a direct summand of $K$, it has a subgroup that fixes some point of $M$.
Therefore, there is a maximal $m'<m$ such that $g' = g^{p^{m'}}$ fixes some point of $M$.
As a consequence, $Fix(g')$ is a smooth, non-necessarily connected, codimension one submanifold of $M$. 
Moreover, $\langle g' \rangle$ acts freely on $M \setminus Fix(g')$.
Thus, for all $x \in Fix(g')$ there is a $\langle g' \rangle$-invariant neighborhood $U$ with local coordinates $z_1,\dots,z_n$ such that $Fix(g') \cap U$ is the locus $z_1=0$, and the action of $\langle g' \rangle$ on $U$ is generated by $z \mapsto (e^{2\pi i/ p^{m-m'}}z_1,z_2,\dots,z_n)$.
Therefore, the quotient $M/\langle g' \rangle$ is smooth and the projection from $M$ is a $\langle g' \rangle$-covering ramified along $Fix(g')$.
Note that $M/\langle g' \rangle$ is acted on freely by the quotient group $\langle g \rangle / \langle g' \rangle$.
Therefore, letting $G_1 = \langle g \rangle$, and $M_1 = M_0/G_1$ yields a cyclic covering $f_1 : M_0 \to M_1$ with Galois group $G_1$.

On the other hand, since $G$ is abelian, the action of $G$ on $M$ descends to an action of $H_1 = G/G_1$ on $M_1$. 
Clearly one has $N = M_1 / H_1$ and, denoted by $p_1 : M_1 \to N$ the projection on the quotient, one can factorize $f = p_1 \circ f_1$.
Note that $p_1$ is a Galois covering with Galois group $H_1$.
Now, the subset of $H_1$ generated by all the isotropy groups for the action of $H_1$ on $M_1$ is isomorphic to $\bigoplus_{g \in S_1} \langle g \rangle$ where $S_1$ is $S$ without the generator of $G_1$.

Applying the same argument as above to the covering $p_1$ one eventually gets a cyclic $G_2$-covering $f_2 : M_1 \to M_2$ and a $H_2$-covering $p_2 : M_2 \to N$, where $H_2 = G / (G_1 \times G_2)$, that satisfies $p_1 = p_2 \circ f_2$. 
Hence $f = p_2 \circ f_2 \circ f_1$.

Let $\ell = |S|$.
Iterating this argument until $j=\ell$, and letting $\hat f = p_\ell$ and $\hat G = H_\ell$, then yields the thesis. 

\end{proof}

When we turn to Galois non-abelian coverings, the possibility of constructing such coverings 
is much more restricted than for cyclic and abelian ones, and indeed it has been subject of intensive research
by many authors for different type of groups (see e.g. \cite{Cat}, \cite{CP}, \cite{FP}, \cite{Man}). In the last subsection of this paper we will focus on the example of the Dihedral group, where thanks to the fundamental work of Catanese-Perroni \cite{CP}, we have a fairly complete understanding. In particular we observe that such covers do exist in abundance even with very singular branch divisors.

\subsection{KE and cscK metrics: the abelian case}

Given the great flexibility in constructing cyclic Galois covers described above, it is also clear that the range of applicability of Theorem \ref{thm:cyclic} is extremely large.

Let us first observe that we can recover immediately the known results for K\"ahler-Einstein manifolds and cyclic covers
proved in Theorem 6.1 in \cite{LS} (see also Theorem 2.5 in \cite{AGP}).

Indeed, let us assume that $N$ is Fano with K\"ahler-Einstein metric $\om_N$ and assume $B\in |-\beta K_N|$
 (note that this $\beta$ is not the one in \cite{AGP}). It is then immediate to check that Theorems \ref{thm:cyclic} and \ref{thm:Galois}
apply if $\beta \geq 1$, hence providing K\"ahler-Einstein metrics on the covering space $M$. Note that $M$ will be Fano again iff $d< \frac{\beta}{\beta -1}$ (which is the condition appearing in Theorem 2.5 of \cite{AGP}).
This application extends  the range of application of Theorem 6.1 in \cite{LS} where $B$ is supposed to be  smooth and irreducible (and similar results in \cite{AGP}).

Even in this simplest situation it is worth observing how our results extend Corollary 2.7 in  Cheng-Chen \cite{CC}:
let us indeed consider as a simple example cyclic coverings of $N=\pr^n$, $\mid G\mid = d$ and 
$B\in|\lambda H|$, where $H$ is the hyperplane bundle.
In this case, Corollary \ref{coro:cyclic} applies.

Our first condition reads:
$$K_N+(1-\beta_0)B = -(n+1)H + \lambda (1-\beta_0) H \geq 0$$

As for the second one, having fixed $\omega_N = \mu \omega_{FS}$, i.e. $[\omega_N]=\mu c_1(H)$, it is immediate to check that we get again the same condition as above.

Therefore, we can conclude that
$$\lambda > n+1 $$
in order to get a cscK metric on M.
(Note that by Theorem \ref{thm:cyclic} one gets the better condition $\lambda \geq n+1$).
On the other hand, Cheng-Chen require
$$ \lambda \geq (n+1) \frac{d}{d-1} .$$

In particular for all cyclic coverings of $\pr^n$ with $d \leq n+2$ we can find an integer $\lambda$
which falls into the range of our results and not in Cheng-Chen's.

In fact, Chen-Cheng's Corollary 2.7 is simply a combination of their main theorem and previous works (e.g. \cite{LSY}) on the properness of K-energy. It is easy to check that a direct application of Chen-Cheng in our setting gives precisely the $\delta=1$ case of Theorem \ref{thm:Galois}. So our theorems are both theoretical and practical improvements of Chen-Cheng's Corollary 2.7 in our setting.

 Let us now illustrate one simple example, where the Ka\"ahler class on the base manifold is not taken 
to be necessarily KE and yet produces via our Theorems also new KE metrics:
fix as base space 
$$(N, \omega_N) = (\pr^1 \times \pr^1, \omega), $$ where $\omega \in c_1(L), \,\, L= a_1 C_1 + a_2 C_2$, $a_1, a_2 >0$,  and $C_i$ denotes the homology class of the $i$-th factor $\pr^1$.

Set $B\in|d(b_1C_1 + b_2C_2)|$ the (smooth) branching divisor. In order for this data to satisfy our setup we need to assume that $b_i >0$. It is then an elementary computation to check that conditions of Corollary \ref{coro:cyclic} are satisfied iff $db_1>2$ and $db_2 >2$.
Moreover, whenever $(a_1,a_2)$ is proportional to $(b_1,b_2)$, the conditions of Theorem \ref{thm:cyclic} reduce to $db_1 \geq 2$, $db_2 \geq 2$, thus adding the case $d=2$, $b_1=b_2=1$, $a_1=a_2$ to the ones given by Corollary \ref{coro:cyclic}.

Let us remark on this example that the scalar curvature of the cscK metric $\omega_M$ is also easily computed by 
$$ s_M = 2 \cdot  \frac{L\cdot(-K_M)}{L^2} = (2-(d-1)b_1)a_2 + (2-(d-1)b_2)a_1, $$
which is always non positive, except for the trivial case $d=1$ and the case $d=2$, $b_1=b_2=1$, $a_1=a_2$. 
Moreover, $s_M$ vanishes if and only if ($d=3$ and $b_1=b_2=1$) or $d=b_1=b_2=2$.

Note that this construction has $4$ free parameters (plus the group $G$): $b_1, b_2$ upon which the geometry of the resulting manifold $M$ depends, and $a_1, a_2$ which control the K\"ahler class of $\omega_M$, showing in practice the large flexibility of our construction in this case.
One can check that $\omega_M$ is KE if and only if $(a_1,a_2)$ and $(2-(d-1)b_1,2-(d-1)b_2)$ are linearly dependent.
Note that this implies one can get KE metrics on the covering even starting from cscK not KE on the base (e.g. $a_1=1, a_2=2, b_1=2, b_2=3, d=3$).

\subsection{KE and cscK metrics: the non-abelian case}

As a prototype example, in this Section we focus on the case of Galois coverings with structure group $D_p$, the dihedral group of order $2p$, whose existence has been largely studied in the above mentioned works.
We assume hereafter that $p \geq 3$ so that $D_p$ is not abelian.

As described by Catanese-Perroni, in the theory of $D_p$-Galois coverings  the parameter $B$ is ``a posteriori" determined by the choice of a line bundle $L$ in such a way that there exists sections 
$a\in H^0(N, \cO_N(pL))$ and $f\in H^0(N, \cO_N(2L))$, such that the following conditions are satisfied (see
\cite{CP}, Section 6 and Theorem 6.1):
\begin{enumerate}
\item the zero-locus of $a^2 -f^p$ is smooth in the open subset where $f\not= 0$;
\item the divisors $A:=\{ a =0\}$ and $F:= \{  f=0\}$ intersect each other transversely;
\item $A\cap F \not= \emptyset$;
\end{enumerate}

A beautiful result (Theorem 6.1 of \cite{CP}) by  Catanese-Perroni shows how with these data at hand one can construct a $D_p$-Galois covering $\pi_{L} \colon M_{L} 
\rightarrow N$ ramified
along the locus $\{a^2 -f^p =0\} \subset N$. It is interesting to observe that in this case the ramification divisor has $p$ components all with inertia group $\Z _2$,
while the branch locus has one component for odd and two for even $p$ respectively.

We want to apply this result in few significant cases, in order to analyze the possibility of applying our Theorem \ref{thm:Galois}.
From now on $G$ is assumed to be the dihedral group $D_p$.

\begin{enumerate}
    \item 
Let us start again with 
 $N=\pr^n$ and we consider $L=\lambda H$  which implies
$B\in|2p \lambda H|$, where $H$ is the hyperplane bundle. In this case we also assume for simplicity 
$[\omega_N] = c_1(H)$.

This implies that in this situation $B_{\delta} \sim \delta p \lambda H$, hence our first condition reads
$$K_N + B_{\delta} \sim (\delta p \lambda - n - 1)H \geq 0 \,\, ,$$
while the second 

$$\frac{n(K_N+B_{\delta})\cdot [\om_N]^{n-1}}{[\om_N]^n}[\om_N]-(n-1)(K_N+B_{\delta}) \geq 0, $$
hence both reduce to 
$$ \delta p \lambda \geq n+1 \,\, ,$$
for some $\delta \in [1, 2)$.

As for the cyclic case described above, it is easy to observe that there are plenty of situations where our result applies while Chen-Cheng's doesn't (the simplest being $n=2$, $p=3$ and $\lambda = 1$, or $n=p=3$ and $\lambda =1$). 

Yet, all these examples fall in the category of K\"ahler-Einstein manifolds, but again, contrary to Li-Sun, we do not have to assume smoothness and irreducibility of the branch locus. 

\item

In order to construct csck non KE metrics we can apply Catanese-Perroni's theory taking as base manifold $\pr^n\times \pr^n$ with $L =  H_1 + H_2$, having denoted by $H_i$ the hyperplane bundle on the $i$-th factor.
We also choose $\omega_N \in b_1 c_1(H_1) + b_2 c_1(H_2)$. 

The first condition

$$K_N + B_{\delta} \sim (\delta p  - n - 1)H_1  +  (\delta p  - n - 1)H_2 \geq 0$$

translates into $\delta p  \geq n + 1$, while the second 

$$\frac{2n(K_N+B_{\delta})\cdot [\om_N]^{2n-1}}{[\om_N]^{2n}}[\om_N]-(2n-1)(K_N+B_{\delta}) \geq 0 $$

becomes 

$$(n \frac{b_1}{b_2} - (n-1))(\delta p  - n - 1)H_1 + (n \frac{b_2}{b_1} - (n-1))(\delta p  - n - 1)H_2 \geq 0$$.

We then have to require either 
$$\delta p = n+1 
\qquad \mbox{or} \qquad
\left\{\begin{array}{l} 
\delta p  >  n+1 \\ 
n \frac{b_1}{b_2} - (n-1) \geq 0 \\ 
n \frac{b_2}{b_1} - (n-1) \geq 0
 \end{array}\right. $$
for some $\delta \in [1,2)$.
Equivalently, either $1 \leq \frac{n+1}{p} < 2$ or 

$$ \left\{\begin{array}{l}
p > \frac{n+1}{2} \\
\frac{n-1}{n}b_2 \leq b_1 \leq \frac{n}{n-1} b_2
\end{array}\right. $$
Note that the above system gives large families of new cscK manifolds for any $n$ (not KE for $b_1 \neq b_2$) and these were not covered by Chen-Cheng's Theorem for $n\geq3$.

\item

Let us end this Section by applying our results and Catanese-Perroni's construction for dihedral $D_p$-Galois covering to the case when $N= Bl_{P_1,P_2,P_3}\pr^2$ is the the projective plane blown-up at three non-collinear points $P_i$.
We denote by $\sigma \colon N \to \pr^2$ the blow-up at $\{ P_1, P_2, P_3\}$, and
$$
L= kH - \sum_{i=1}^3 a_iE_i \, , 
$$
where $H$ is the pull-back via $\sigma$ of any line not passing through the points $P_i$'s, $E_i = \sigma^{-1}(P_i)$ is the exceptional curve over $P_i$ for all $i=1,2,3$, and $k,a_1,a_2,a_2$ are integers.

Since $Bl_{P_1,P_2,P_3}\pr^2$ is toric, $L$ is very ample if and only if it is ample \cite[Theorem 6.1.15]{CoxLittleSchenck}.
On the other hand, whenever $2L$ is very ample, a simple application of Bertini's Theorem gives the existence of the seeked holomorphic sections $a\in H^0(N, \cO_N(pL))$ and $f\in H^0(N, \cO_N(2L))$ which satisfy the above requirements (the transversality property required follows at once from \cite[Lemma V,1.2]{H}).
Summarizing, in order to apply \cite{CP}, the requirement is just that $L$ is ample. 
As a strightforward application of the Nakai-Moishezon Criterion (see e.g. \cite{H}) this is equivalent to
$$a_1, a_2, a_3 >0, \qquad\mbox{and}\qquad k>a_i+a_j \quad \forall i\neq j.$$

%


%

Having now $\pi_{L} \colon M_{L} \rightarrow N$, branched along $B \in | 2pL|$, we have to choose the cscK class $[\omega_N]$. We can make (at  least) two very different, but equally interesting, choices:

\begin{itemize}
    \item 
    $[\omega_N] = -K_N$ and $\omega_N$ is the K\"ahler-Einstein metric found by Siu \cite{Siu1988} and Tian-Yau \cite{TianYau1987}.
    
    \item 
    $[\omega_N] = F$ where $F =  r H - \sum_{i=1}^3 E_i$, with $r$ positive and
    sufficiently large so that $F$ contains a cscK metric via Theorem 1.3 in  \cite{AP}.
\end{itemize}

Clearly, we can treat these cases simultaneously denoting by $[\omega_N] =r H - \sum_{i=1}^3 E_i$,
where $r = 3$ gives the Siu-Tian-Yau case, while $r \gg 3$ gives the Arezzo-Pacard one.

When is $K_N+B_{\delta}$ nef?

$$K_N+B_{\delta} = (\delta p k -3)H - \sum_{i=1}^3(\delta p a_i - 1)E_i,$$
hence the nef condition translates into $ k \geq (a_i + a_j) + \frac{1}{\delta p}$ for all $i\neq j$.
Equivalently, thanks to the fact that $k$ and $a_i$'s are integers, $p \geq 3$, and $1 \leq \delta < 2$, nef condition reduces to
\begin{equation} k \geq a_i + a_j + 1 \qquad \forall i\neq j. \end{equation}

When is $\frac{2(K_N+B_{\delta})\cdot [\om_N]}{[\om_N]^2}[\om_N]-(K_N+B_{\delta})$ nef?

Notice that
$$\frac{2(K_N+B_{\delta})\cdot [\om_N]}{[\om_N]^2}[\om_N] = 2 \frac{r(\delta p k -3) - \sum_{i=1}^3(\delta p a_i - 1)}{r^2 - 3} (rH -  \sum_{i=1}^3  E_i).$$



Given the number of parameters involved let us first fix $r=3$ hence looking at Siu-Tian-Yau's KE metric on the base.
In this case $\frac{2(K_N+B_{\delta})\cdot [\om_N]}{[\om_N]^2}[\om_N]-(K_N+B_{\delta})$ reduces to

$$ [\delta p (2 k - \sum_{j=1}^3a_j) -3]H 
- \sum_{i=1}^3  [\delta p (k - a_i - \frac{1}{3}\sum_{j=1}^3 a_j ) - 1] E_i$$
hence the nef condition becomes $ \delta p( a_i+a_j - \frac{1}{3}\sum_{\ell=1}^3 a_\ell) \geq 1$ for all $i\neq j$, or equivalently
$ a_i \leq -\frac{1}{\delta p} + \frac{2}{3}\sum_{\ell=1}^3 a_\ell$ for all $i \in \{1,2,3\}$.
Since all $a_i$'s are integers and $p \geq 3$, this condition is satisfied for some $\delta \in [1,2)$ whenever
$$ a_i \leq \frac{1}{3}\sum_{\ell=1}^3 (2a_\ell - 1), \quad \forall i \in \{1,2,3\} . $$

As a consequence, as soon $a_i$'s satisfy this condition and $k$ is chosen so that $L$ is ample, $M$ admits a cscK metric in $\pi_L^* (-K_N)$ thanks to Theorem \ref{thm:Galois}.

Going back to the adiabatic cscK classes given by $r \gg 3$, let us analyze for simplicity the case $a_i = 1$ for all $i$, yet is clearly not a necessary condition for our Theorem to apply.

In this case, $\frac{2(K_N+B_{\delta})\cdot[\om_N]}{[\om_N]^2}[\om_N]-(K_N+B_{\delta})$ becomes 

$$ \frac{\delta p (r^2 k + 3 k - 6 r) - 3(r^2 - 2r +3)}{r^2 - 3}H 
- \sum_{i=1}^3  \frac{\delta p (2r k + 3 - r^2 - 6 )+r^2 - 6r + 3}{r^2 - 3} E_i$$
and it is nef if and only if the coefficient of $H$ is not smaller than the coefficient of $E_i$, or equivalently
$$ (k+2) r^2 - (4k+6) r + 3 k + 6 \geq \frac{5r^2 - 18r + 15}{\delta p}.$$
This inequality is satisfied for some $\delta \in [1,2)$ whenever $k\geq 4$ and $r \gg 3$, hence providing new large families of csck manifolds by Theorem \ref{thm:Galois}.

Finally notice that in all cases discussed above, under assumption of generality of the parametrers involved, cscK metric constructed on $M$ are not KE.
In order to see this, note that by Catanese-Perroni's construction one has $K_M = \pi_L^*(K_N + pL)$.
On the other hand, the cscK metric we construct over $M$ represents the class of $\pi^*F $, thus it is KE if and only if there is a real constant $\lambda$ such that $K_N+pL = \lambda F$.
In turns, this is equivalent to
$$ (pk-3-\lambda r)H - \sum_{i=1}^3 (pa_i-1-\lambda)E_i = 0, \quad \mbox{for some} \quad \lambda \in \mathbf R. $$
Therefore, in order to make sure that the cscK metric we constructed on $M$ is not KE, the parameters must satisfy:
\begin{itemize}
\item $k \neq 3a_i$ for some $i$  (Siu-Tian-Yau case, i.e. $r=3$),
\item $k \neq \frac{3+(pa_i-1)r}{p}$ for some $i$ (Arezzo-Pacard case, i.e. $r \gg 3$).
\end{itemize}
It is then clear that the cscK metric we constructed on $M$ is never KE in the Arezzo-Pacard case up to taking $r$ sufficiently large, once $k$ and the $a_i$'s are fixed.
On the other hand, that metric is KE in the Siu-Tian-Yau case if and only if $k=3a_1=3a_2=3a_3$.

We end up this subsection by noting that for all coverings of $Bl_{P_1,P_2,P_3}\pr^2$ discussed here does not apply a result of Chen-Cheng stating that any complex surface with negative first Chern class and no curves with negative self-intersection admits a cscK metric in any K\"ahler class \cite[Corollary 1.7]{CC}.
Indeed, our examples are surfaces having negative first Chern class since $K_M = \pi_L^*(K_N+pL)$, and
$$ K_N+pL = (p k -3)H - \sum_{i=1}^3(p a_i - 1)E_i $$
is ample since $k > a_i+a_j+\frac{1}{p}$ for all $i \neq j$.
On the other hand, the finiteness of the map $\pi_L$ implies that pulling back any $(-1)$-curve of $Bl_{P_1,P_2,P_3}\pr^2$ yields a curve with negative self-intersection on $M$.
\end{enumerate}


\bigskip


\bigskip
\begin{thebibliography}{99}

\footnotesize

\bibitem{AGP} Arezzo, C., Ghigi, A., and Pirola, G.  {\em Symmetries, quotients and K\"ahler-Einstein metrics}, J. reine angew. Math. 591 (2006), 177-200.

\bibitem{AP} Arezzo, C., Pacard, F., {\em Blowing up K\"ahler manifolds with constant scalar curvature. II.}
Ann. of Math. (2) 170 (2009), no. 2, 685–738.

\bibitem{Ber} Berman, R. {\em A thermodynamical formalism for Monge-Amp\`ere equations, Moser-Trudinger inequalities and K\"ahler-Einstein metrics.} Adv. Math. 248 (2013), 1254-1297.

\bibitem{BDL}Berman, R., Darvas, T., and Lu, C. {\em Convexity of extended K-energy and the large time behavior of the weak Calabi flow}, Geometry and Topology, 21 (2017), 2945-2988.

\bibitem{Cat} Catanese, F. {\em On the moduli spaces of surfaces of general type}, J. Differential Geometry, 19 (1984), 483-515.

\bibitem{CP} Catanese, F., Perroni, F., {\em Dihedral Galois covers of algebraic varieties and the simple cases}, J. Geom. Phys. 118 (2017), 67-93.

\bibitem{Chen} Chen, X.X. {\em On the existence of constant scalar curvature K\"ahler metrics: a new perspective}, Ann. Math. Qu\'e. 42 (2018), no. 2, 169-189.

\bibitem{CC} Chen, X.X. and Cheng, J.  {\em On the constant scalar curvature K\"ahler metrics (II)-Existence results}, J. Amer. Math. Soc. 34 (2021), no. 4, 937-1009.

\bibitem{CoxLittleSchenck} Cox, D. A., Little, J. B., and Schenck, H. K.
{\em Toric varieties}.
Graduate Studies in Mathematics, 124. American Mathematical Society, Providence, RI, 2011.

\bibitem{Dar} Darvas, T. {\em The Mabuchi geometry of finite energy classes}, Adv. Math. 285 (2015), 182-219.

\bibitem{DarRub} Darvas, T. and Rubinstein, Y. {\em Tian's properness conjectures and Finsler geometry of the space of K\"ahler metrics}, Journal of the AMS. 30 (2017), 347-387.

\bibitem{DeP} Demailly, J.P. and Paun, M. {\em Numerical characterization of the K\"ahler cone of a compact K\"ahler manifold}, Ann. of Math. (2) 159 (2004), no. 3, 1247-1274.


\bibitem{Der2} Dervan, R. {\em On K-stability of finite covers}, Bull. Lond. Math. Soc. 48 (2016), no. 4, 717-728.

\bibitem{FP} Fantechi B. and Pardini, R., {\em Automorphisms and moduli spaces of varieties with ample canonical class via deformations of abelian covers}. Comm. Algebra 25 (1997), no. 5, 1413–1441.

\bibitem{GhKo} Ghigi, A. and Koll\'ar, J. {\em K\"ahler-Einstein metrics on orbifolds and Einstein metrics on spheres}, Comment. Math. Helv. 82 (2007), 877-902.

\bibitem{KMM} Y. Kawamata, K. Matsuda and K. Matsuki,
{\em Introduction to the Minimal Model Problem},
 Advanced Studies in Pure Mathematics 10, 1987 Algebraic Geometry, Sendai, 1985

\bibitem{H} Hartshorne, R., {\em Algebraic Geometry}, Springer-Verlag, Berlin, 1977.

\bibitem{Has} Hashimoto, Y. {\em Existence of twisted constant scalar curvature K\"ahler metrics with a large twist}, Math. Zeit. 292 (2019), 791-803.

\bibitem{La} R. Lazarsfeld. {\em Positivity in Algebraic Geometry, I}, Springer-Verlag, Berlin, 2004.

\bibitem{Man} M. Manetti {\em Iterated double covers and connected components of moduli spaces.} Topology 36 (1997), no. 3, 745–764.

\bibitem{LS} Li, C. and Sun, S. {\em Conical K\"ahler-Einstein metrics revisited},Comm. Math. Phys. 331 (2014), no. 3, 927-973.

\bibitem{LSY} Li, H., Shi, Y. and Yao, Y. {\em A criterion for the properness of the K-energy in a general K\"ahler class}, Math. Ann. 361 (2015), no. 1-2, 135-156.

\bibitem{P} Pardini, R.,
{\em  Abelian covers of algebraic varieties}, J. Reine Angew. Math. 417 (1991), 191-213. 

\bibitem{SD} Sj\"ostr\"om Dyrefelt, Z. {\em Optimal lower bounds for Donaldson's J-functional}, Advances in Mathematics, 374 (2020), 107271.

\bibitem{SW} Song, J. and Weinkove, B. {\em On the convergence and singularities of the J-flow with applications to the Mabuchi energy}, Comm. Pure Appl. Math. 61 (2008), no. 2, 210-229.

\bibitem{Siu1988} Siu, Y. T. {\em The existence of K\"ahler-Einstein metrics on manifolds with positive anticanonical line bundle and a suitable finite symmetry group}. Ann. of Math. (2) 127 (1988), no. 3, 585--627.

\bibitem{T0} Tian, G. {\em On K\"ahler-Einstein metrics on certain K\"ahler manifolds with $C_1(M)>0$}, Invent. Math. 89 (1987), no. 2, 225--246.

\bibitem{TianYau1987} Tian, G. and Yau, S.-T. {\em K\"ahler-Einstein metrics on complex surfaces with $C_1>0$}. Comm. Math. Phys. 112 (1987), no. 1, 175--203.

\bibitem{Zeng} Zeng, Y. {\em Deformations from a given K\"ahler metric to a twisted cscK metric}, Asian J. Math. 23 (2019), no. 6, 985-1000.

\bibitem{Zha} Zhang, K., {\em Continuity of delta invariants and twisted K\"ahler-Einstein metrics}, Advances in Mathematics, 388 (2021), 107888.

\bibitem{Zh} Zheng, K. {\em Existence of constant scalar curvature Kaehler cone metrics, properness
and geodesic stability,} arXiv:1803.09506.




\end{thebibliography}
\end{document}